\documentclass[12pt]{amsart}
\usepackage{amsmath}
\usepackage{amssymb}
\usepackage{amsthm}
\usepackage{bm}

\numberwithin{equation}{section}
\newtheorem{theorem}{Theorem}[section]
\newtheorem{lemma}[theorem]{Lemma}
\newtheorem{proposition}[theorem]{Proposition}

\newtheorem{example}[theorem]{Example}
\newtheorem{definition}[theorem]{Definition}

\newtheorem{remark}[theorem]{Remark}

\newcommand{\al}{\alpha}

\newcommand{\gm}{\gamma}

\newcommand{\ep}{\varepsilon}

\newcommand{\om}{\omega}

\newcommand{\q}{\quad}
\newcommand{\qq}{\qquad}

\newcommand{\wh}{\widehat}

\newcommand{\R}{\mathbb{R}}
\newcommand{\N}{\mathbb{N}}
\newcommand{\Z}{\mathbb{Z}}
\newcommand{\Prime}{\mathbb{P}}
\newcommand{\Q}{\mathbb{Q}}

\newcommand{\rd}{\mathbb R ^d}

\newcommand{\pn}{\par\noindent}

\newcommand{\mZ}{\mathcal{Z}}

\newcommand{\va}{{\Vec{a}}}

\newcommand{\vs}{{\Vec{s}}}
\newcommand{\vt}{{\Vec{t}}}
\newcommand{\vsig}{{\Vec{\sigma}}}

\allowdisplaybreaks
\begin{document}
\title[Multidimensional polynomial Euler products]{Multidimensional polynomial Euler products and infinitely divisible distributions on $\rd$}
\author[T.~ Aoyama and T.~Nakamura]{Takahiro Aoyama$^{*}$ \& Takashi Nakamura$^{**}$}
\address{$^{*}$ Department of Environmental and Mathematical Sciences, Fuculty of Environmental Science and Technology, Okayama University, 
2-1-1 Tsushima-naka, Kita-ku, Okayama-City, 700-8530, Japan}
\email{{taoyama$@$okayama-u.ac.jp}}
\address{$^{**}$ Department of Liberal Arts,
Faculty of Science and Technology, Tokyo University of Science,
2641 Yamazaki, Noda-shi, Chiba-ken, 278-8510, Japan}
\email{{nakamuratakashi@rs.tus.ac.jp}}

\maketitle

\begin{abstract}
It is known to be difficult to find out whether a certain multivariable function to be a characteristic function 
when its corresponding measure is not trivial to be or not to be a probability measure on $\rd$.
Such results were not obtained for a long while.
In this paper, multidimensional polynomial Euler product is defined as a generalization of the polynomial Euler product.
By applying the Kronecker's approximation theorem, 
a necessary and sufficient condition for some polynomial Euler products to generate characteristic functions is given.
Furthermore, by using the Baker's theorem, that of some multidimensional polynomial Euler products is also given.
As one of the most important properties of probability distributions, the infinite divisibility of them is studied as well.
\end{abstract}

\tableofcontents 

\section{Introduction}
\subsection{Infinitely divisible distributions}

Infinitely divisible distributions are known as one of the most important class of distributions in probability theory.
They are the marginal distributions of stochastic processes having independent and stationary increments such as Brownian motion and Poisson processes. 
In 1930's, such stochastic processes were well-studied by P. L\'evy and now we usually call them L\'evy processes.
We can find the detail of L\'evy processes in \cite{S99}.

In this section, we mention some known properties of infinitely divisible distributions.

\begin{definition}[Infinitely divisible distribution]
A probability measure $\mu$ on $\rd$ is infinitely divisible if,
for any positive integer $n$, there is a probability measure $\mu_n$ on 
$\rd$
such that
\begin{equation*}
\mu=\mu_n^{n*},
\end{equation*}
where $\mu_n^{n*}$ is the $n$-fold convolution of $\mu_n$.
\end{definition}

\begin{example}\label{exa}
Normal, degenerate, Poisson and compound Poisson distributions are infinitely divisible.
\end{example}

Denote by $I(\rd)$ the class of all infinitely divisible distributions on $\rd$.
Let $\wh{\mu}(\vt):=\int_{\rd}e^{{\rm i}\langle \vt,x\rangle}\mu (dx),\, \vt\in\rd,$ be the characteristic function of a distribution $\mu$, 
where $\langle\cdot ,\cdot\rangle$ is the inner product.
We also write $a\wedge b=\min \{a,b\}$.
The following is well-known.

\begin{proposition}[L\'evy--Khintchine representation (see, e.g.\,\cite{S99})]\label{pro:LK1}
$(i)$ If $\mu\in I(\rd)$, then
\begin{equation}\label{INF}
\wh{\mu}(\vt) = \exp\left[-\frac{1}{2} \langle \vt,A\vt \rangle+{\rm i}
\langle \gm,\vt \rangle+\int_{\rd}\left(e^{{\rm i} \langle \vt,x \rangle }-1-
\frac{{\rm i} \langle \vt,x \rangle}{1+|x|^2}\right)\nu(dx)\right],\,\vt\in\rd,
\end{equation}
where $A$ is a symmetric nonnegative-definite $d \times d$ matrix,
$\nu$ is a measure on $\rd$ satisfying
\begin{equation}\label{lev}
\nu(\{0\}) = 0 \,\,and\,\,\int_{\rd} (|x|^{2} \wedge 1) \nu(dx) < \infty,
\end{equation}
and $\gm\in\rd$.

$(ii)$ The representation of $\wh{\mu}$ in $(i)$ by $A, \nu,$ and $\gm$ is
unique.

$(iii)$ Conversely, if $A$ is a symmetric nonnegative-definite $d\times d$
matrix, $\nu$ is a measure satisfying $\eqref{lev}$, and $\gm\in\rd$, then
there exists an infinitely divisible distribution $\mu$ whose characteristic function is given by $\eqref{INF}$.
\end{proposition}

The measure $\nu$ and $(A, \nu ,\gm)$ are called the L\'evy measure and the L\'evy--Khintchine triplet of $\mu\in I(\rd)$, respectively. There is another form of $\eqref{INF}$ if the L\'evy measure $\nu$ satisfies an additional condition.

\begin{proposition}[see, e.g.\,\,\cite{S99}]\label{pro:LK2}
In Proposition \ref{pro:LK1}, if the L\'evy measure $\nu$ in $\eqref{INF}$ satisfies an additional condition $\int_{|x|<1}|x|\nu(dx)<\infty$, then we can rewrite the representation $\eqref{INF}$ by
\begin{equation}\label{INF2}
\wh{\mu}(\vt) = \exp\left[-\frac{1}{2} \langle \vt,A\vt \rangle+{\rm i}
\langle \gm_0,\vt \rangle+\int_{\rd}\left(e^{{\rm i} \langle \vt,x \rangle }-1\right)\nu(dx)\right],\,\vt\in\rd,
\end{equation}
where $\gm_0=\gm-\int_{|x|<1}x\left(1+|x|^2\right)^{-1}\nu(dx)$.
\end{proposition}

In this paper, the following form of characteristic functions which are called compound Poisson is to be appeared.

\begin{example}\label{examples}
Let $\mu_{{\rm CPo}}$ be a compound Poisson distribution.
Then, for some $c>0$ and for some distribution $\rho$ on $\rd$ with $\rho (\{0\})=0$,
\begin{equation}\label{CPcf}
\wh\mu_{{\rm CPo}}(\vt)=\exp \left(c\left(\wh\rho (\vt) -1\right)\right), \,\,\vt\in\rd. 
\end{equation}
The Poisson distribution is a special case when $d=1$ and $\rho =\delta_1$.
\end{example}

The following are important facts which play key roles in this paper.

\begin{proposition}[see, e.g.\,\cite{S99}]\label{pro:Sato}
Let $\mu$ be a probability measure on $\rd$.\\
$(i)$ We have that $\wh\mu$ is uniformly continuous, $\wh\mu(0)=1$ and $|\wh\mu (\vt)|\le 1$ for any $\vt\in\rd$.\\
$(ii)$ Let $n$ be a positive even integer. 
If $\wh\mu(\vt)$ is of class $C^n$ in a neighborhood of the origin, then $\mu$ has finite absolute moment of order $n$.\\ 
$(iii)$ If $\mu\in I(\rd)$, then $\wh\mu$ does not have zeros that is $\wh\mu (\vt)\neq 0$ for any $\vt\in\rd$.
\end{proposition}
 
\subsection{Riemann zeta function and Euler Products}

Zeta functions are one of the most important object in number theory. 
The Riemann zeta function is regarded as the origin and it is defined as follows.

\begin{definition}[Riemann zeta function (see, e.g.\,\cite{Apo})]
The Riemann zeta function is a function of a complex variable $s = \sigma + {\rm i}t$, for $\sigma >1$ given by
\begin{align}
\zeta (s) := \sum_{n=1}^{\infty} \frac{1}{n^s}
 = \prod_p \Bigl( 1 - \frac{1}{p^s} \Bigr)^{-1} ,\label{eq:eupro}
\end{align}
where the letter $p$ is a prime number,  and the product of $\prod_p$ is taken over all primes.
\end{definition}
The infinite series and product are called Dirichlet series and Euler product, respectively.
Both of the Dirichlet series and the Euler product of $\zeta (s)$ converges absolutely in the half-plane $\sigma >1$ and uniformly in each compact subset of this half-plane. 

As one of a generalization of $\zeta (s)$, the Hurwitz zeta function is also well-known.
Its definition is as follows.
For $0 < u \le 1$ and $\sigma >1$, the Hurwitz zeta function $\zeta(s,u)$ is defined by
\begin{equation}\label{def:ler}
\zeta (s,u) := \sum_{n=0}^{\infty} \frac{1}{(n+u)^s}.
\end{equation}
Note that we obviously have $\zeta (s) = \zeta (s,1)$. 
The function $\zeta (s,u)$ is analytically continuable to the whole complex plane as a meromorphic function with a simple pole at $s=1$. 

Next we introduce Dirichlet characters and the Dirichlet $L$-functions. 
Let $q$ be a positive integer. 
A Dirichlet character $\chi$ mod $q$ is a non-vanishing group homomorphism from the group $({\mathbb{Z}}/q{\mathbb{Z}})^*$ of prime residue classes modulo $q$ to ${\mathbb{C}}^*$. 
The character which is identically one is denoted by $\chi_0$ and is called the principal. 
By setting $\chi (n) = \chi (a)$ for $n \equiv a \mod q$, we can extend the character to a completely multiplicative arithmetic function on ${\mathbb{Z}}$. 
\begin{definition}[Dirichlet $L$-function (see, e.g.\,\cite{Apo})]
For $\sigma >1$, the Dirichlet $L$-function $L(s,\chi)$ attached to a character $\chi$ mod $q$ is given by
\begin{equation}
L (s,\chi) := \sum_{n=1}^{\infty} \frac{\chi (n)}{n^s} = \prod_p \Bigl( 1 - \frac{\chi(p)}{p^s} \Bigr)^{-1} .
\end{equation}
\end{definition}
The Riemann zeta function $\zeta (s)$ may be regarded as the Dirichlet $L$-function to the principal character $\chi_0 \mod 1$. 
It is possible that for values of $n$ coprime with $q$ the character $\chi (n)$ may have a period less than $q$. 
If so, we say that $\chi$ is imprimitive, and otherwise primitive. 
Every non-principal imprimitive character is induced by a primitive character. 
Two characters are non-equivalent if they are not induced by the same character. 
Characters to a common modulus are pairwise non-equivalent. 

The exact analytic translation of the existence and uniqueness of prime decomposition in $\Z$ is 
the fact that the Dirichlet series associated with a completely multiplicative function has an Euler product. If $K$ is a general number field and $\Z_K$ its ring of integers, the existence and uniqueness of prime decomposition are valid for ideals as in the case for every Dedekind domain. 
Thus, it is natural to define the following function
$$
\zeta_K (s) := \sum_{{\mathfrak{a}} \subset \Z_k} \frac{1}{{\mathcal{N}} ({\mathfrak{a}})^s}
= \prod_{{\mathfrak{p}}} \Bigl( 1 - \frac{1}{{\mathcal{N}} ({\mathfrak{a}})^s} \Bigr)^{-1}, \qq  \sigma >1
$$
where ${\mathfrak{a}}$ runs through all integral ideals of $\Z_K$ and ${\mathfrak{p}}$ through all prime ideals of $\Z_K$ and ${\mathcal{N}}$ denotes the absolute norm. 
The function $\zeta_K$ is called the Dedekind zeta function.
The equality of the two definitions is the exact translation of existence and uniqueness of prime ideal decomposition. 
We of course have $\zeta_{\Q} (s) = \zeta (s)$. 
As mentioned in \cite[Proposition 10.5.5]{Cohen}, let $K=\Q (\sqrt{D})$ be a quadratic field of discriminant $D$.
Then we have $\zeta_K (s) = \zeta (s) L(s,\chi_D)$, where $\chi_D$ is the Legendre--Kronecker character. 
Moreover, we have the following (see, \cite[Theorem 10.5.22]{Cohen}). 
Let $\Q_m$ be the $m$-th cyclotomic field. We have $\zeta_{\Q_m} (s) = \prod_{\chi \,\, {\rm{mod}} \,\, m} L(s,\chi_f)$, where $\chi_f$ is the primitive character associated with $\chi$. 
In particular, if $m$ is a prime power we have $\zeta_{\Q_m} (s) = \prod_{\chi \,\, {\rm{mod}} \,\, m} L(s,\chi)$. 

These famous functions may be regarded as the prototype of zeta functions which have Euler products. 
Many authors have introduced classes of {\textit{L}}-functions to find common patterns of their value-distributions. 
The most well-known class ${\mathcal{S}}$ was introduced by Selberg \cite{Selberg}. 
His aim was to study the value-distribution of linear combinations of $L$-functions. 
In the meantime, this so-called Selberg class became one of the important objects in number theory.
However, there still remain unknown properties. 
Afterwards, Steuding \cite{Steuding1} introduced a class of Dirichlet series satisfying several quite natural analytic axioms with two arithmetic conditions added, namely, a polynomial Euler product representation 
$$
\prod_p \prod_{l=1}^m \Bigl( 1 - \frac{\alpha_l (p)}{p^s} \Bigr)^{-1}, \qquad \sigma >1,
$$
where $|\alpha_l(p)|\le 1$ for $1\le l\le m$ and prime numbers $p$, and has given some kind of the prime number theorem.

On the other hand, Bhowmik, Essouabri and Lichtin  \cite{Bho} extended classical one-variable results about the Euler products, 
defined by integral-valued polynomial or analytic functions, to multi-variable ones. 

\subsection{Zeta distributions on $\R$}

In probability theory, there exists a class of distribution on $\R$ which is generated by the Riemann zeta function.
First it is introduced in \cite{JW35} and we can also find it in \cite{GK68} and \cite{Khi}.
What we have known about this distribution is not many.
Few properties are noted in \cite{GK68} and further ones are studied in \cite{Lin}.
In this section, we mention the Riemann and Hurwitz zeta distributions with some known properties.

\begin{definition}[Riemann zeta distribution on $\R$]
Let $\sigma >1$. A distribution $\mu_{\sigma}$ on $\R$ is said to be a Riemann zeta distribution with parameter $\sigma$ if, for any $n\in\N$, 
$$
\mu_{\sigma}(\{-\log n\})=\frac{n^{-\sigma}}{\zeta (\sigma)}.
$$
\end{definition}

Put
\begin{equation*}
f_{\sigma}(t):=\frac{\zeta (\sigma +{\rm i}t)}{\zeta (\sigma)}, \qq t\in\R,
\end{equation*}
then $f_{\sigma}(t)$ is known to be a characteristic function of $\mu_{\sigma}$ (see, e.g.\,\cite{GK68}).

The Riemann zeta distribution is known to be infinitely divisible.
The characteristic functions and L\'evy measures of them can be given of the form as in the following.

\begin{proposition}[see, e.g.\,\cite{GK68}]\label{pro:RD}
Let $\mu_{\sigma}$ be a Riemann zeta distribution on $\R$ with characteristic function $f_{\sigma}(t)$.
Then, $\mu_{\sigma}$ is compound Poisson on $\R$ and
\begin{align*}
\log f_{\sigma}(t)=
\sum_{p}\sum_{r=1}^{\infty}\frac{p^{-r\sigma}}{r}\left(e^{-{\rm i}rt\log p}-1\right)
=\int_0^{\infty}\left(e^{-{\rm i}tx}-1\right)N_{\sigma}(dx),
\end{align*}
where $N_{\sigma}$ is given by
\begin{align*}
N_{\sigma}(dx)=\sum_{p}\sum_{r=1}^{\infty}\frac{p^{-r\sigma}}{r}\delta_{r\log p}(dx),
\end{align*}
where $\delta_x$ is the delta measure at $x$.
\end{proposition}

As a generalization, the Hurwitz zeta case is also known.
Put a normalized function as follows:
\begin{equation*}
f_{\sigma,u}(t):=\frac{\zeta (\sigma +{\rm i}t,u)}{\zeta (\sigma,u)}, \q\q t\in\R.
\end{equation*}
Then $f_{\sigma,u}$ is also known to be a characteristic function (see, \cite[Theorem 1]{Hu06}).

The infinite divisibility of $\mu_{\sigma,u}$ is studied in \cite{Hu06}.

\begin{proposition}[{\cite[Theorems 2 and 3]{Hu06}}]\label{proHu1}
The Hurwitz zeta distribution $\mu_{\sigma,u}$ is infinitely divisible if and only if $u=1/2$ or $u=1$. Moreover, the Hurwitz zeta distribution $\mu_{\sigma,\frac{1}{2}}$ is compound Poisson (infinitely divisible) 
with its L\'evy measure $N_{\sigma,\frac{1}{2}}$ given by
\begin{equation*}
N_{\sigma,\frac{1}{2}}(dx)=\sum_{p>2}\sum_{r=1}^{\infty}\frac{p^{-r\sigma}}{r}\delta_{r\log p}(dx),
\end{equation*}
where the first sum is taken over all odd primes $p$.
\end{proposition}

These results are the only known ones in the study of zeta functions and their definable infinitely divisible distributions.

\begin{remark}
We have to note that the Riemann zeta distribution is defined in the region of absolute convergence.
The parameter $\sigma$ is always larger than 1 not in the whole complex plane.
When $\sigma =1$, it is well-known that $\zeta(1+it) \ne 0$, $t \ne 0$, and $\zeta(s)$ has an only one pole at $s=1$. 
Hence we have $\zeta (1+it) / \zeta(1) =0$ for any $t \ne 0$, which contradicts Proposition \ref{pro:Sato} $(i)$. 
Also, it is known that $\zeta(1/2) = -1.460354509\dots$ and there exists $t\in\R$ such that $\zeta(1/2+it)$ takes an absolute value larger than that of $\zeta(1/2)$, for example, $|\zeta(1/2+17{\rm i})|=2.142712183\dots$. 
Thus there exist $t\in\R$ such that $|\zeta(1/2+it) / \zeta(1/2)|> 1$ which also contradicts Proposition \ref{pro:Sato} $(i)$. 
Moreover, by Voronin's denseness theorem $($see, e.g. [19, Theorem1.6]$)$, for any fixed $1/2 < \sigma < 1$, 
there exist $t\in\R$ such that $|\zeta(\sigma+it) / \zeta(\sigma)| > 1$. 
Therefore, normalized functions $\zeta(\sigma+it) /\zeta(\sigma)$ can not be characteristic functions for any $1/2 \le \sigma \le 1$.
\end{remark}

\subsection{Aim}

In Section 1.3, we have mentioned infinite divisibility of some known zeta distributions on $\R$.
Again, as in Proposition \ref{pro:Sato} $(iii)$, we note that infinitely divisible characteristic functions do not have zeros.
In zeta cases, this property can give us information of zeros of zeta functions which is one of the most important subject in number theory.
Historically, there exist many continuous probability distributions on $\rd$ and multiple zeta functions but no infinitely divisible zeta distributions on $\rd$ are introduced.
Here zeta distributions on $\rd$ mean that distributions which are discrete with infinitely many mass points that we usually do not see in the sense of explicitly written ones.

However, it is known to be difficult to define probability distributions on $\rd$ by a certain multivariable function as a characteristic function 
if it is not trivial to be so.
They can be defined when they are unsigned finite measures with total mass $1$.
The method how to check the signs of measures on $\rd$, we have the Bochner's theorem in common (see, e.g. \cite{S99}). 
As recent topics of zeta functions related to probability theory, including the property of infinite divisibility, 
some classical results of zeta functions are interpreted probabilistically in \cite{BY},
and the characteristic polynomial of a random unitary matrix is studied by a probabilistic approach in \cite{Yor}.
The purpose of our recent work is to establish zeta distributions on $\rd$ as new treatable discrete distributions on $\rd$ with infinitely many mass points and show properties of them including the relationship with number theory.
As a first generalization of zeta distributions, we have introduced a new multidimensional Shintani zeta function and corresponding zeta distributions on $\rd$ in \cite{AN11k} and revised in \cite{AN12s} with some new results.
By the series representation, non infinite divisibility of them can be seen but there are no methods to show their infinite divisibility.
As to show that, we have to treat the Euler products.

In this paper, a new generalization of the polynomial Euler product is introduced and some important examples are given in Section 2.
Our main results are in Section 3.
The necessary and sufficient conditions for those products to generate discrete compound Poisson characteristic functions are given and new classes of multidimensional compound Poisson zeta distributions are defined.
These results contain new methods of how to check when some multivariable functions can be characteristic functions, when probability distributions can be defined in other words, by applying the Baker's and Kronecker's approximation theorems which are well-known in number theory.
Still, it seems to be not well-studied even for $1$-dimensional case. 
So that we give some new important examples of $1$-dimensional zeta functions related to these new classes in Section 4.

\section{Multidimensional polynomial Euler Products}

In Section 1.4, we mentioned that our purpose of this paper is to establish a new zeta distributions on $\rd$ by Euler products as new treatable multidimensional
discrete infinitely divisible distributions with infinitely many mass points.
We consider the Euler product case even we have studied the multiple series case as in \cite{AN11k} and, afterwards, in \cite{AN12s} since there are two merits; 
(I) Euler products do not vanish in the region of absolute convergence (see, Proposition \ref{pro:Sato} $(iii)$).  
(II) It is easy to obtain the L\'evy--Khintchine representation (see, Theorems \ref{th:cfep}, \ref{th:cfep3} and \ref{th:main}). 

\subsection{Definition and properties}

Denote by $\Prime$ the set of all prime numbers.

\begin{definition}[Multidimensional polynomial Euler product, $Z_E(\vs)$]\label{def:EP}
Let $d,m\in\N$ and $\vs\in\mathbb{C}^d$.
For $-1 \le \alpha_l(p) \le 1$ and $\va_l \in {\mathbb{R}}^d$, $1\le l\le m$ and $p\in\Prime$,
we define multidimensional polynomial Euler product given by
\begin{equation}
Z_E (\vs) = \prod_p \prod_{l=1}^m \left( 1 - \alpha_l(p) p^{-\langle \va_l,\vs\rangle} \right)^{-1}.
\label{eq:def1}
\end{equation}
\end{definition}

\begin{remark}\label{rem:con}
We have $Z_E \left(\overline{\vs}\right)=\overline{Z_E (\vs)}$, where $\overline{\Vec{z}}$ is the complex conjugate of $\Vec{z}\in\mathbb{C}^d$.
\end{remark}
Then this product converges absolutely.

\begin{theorem}\label{th:EPc} 
The product \eqref{eq:def1} converges absolutely and has no zeros in the region $\min_{1\le l\le m}\Re \langle \va_l,\vs\rangle >1$. 
\end{theorem}

For the proof of this theorem, we use the following proposition. 
Let $G$ be an open subset of the complex plane and $H(G)$ be the collection of analytic functions on $G$. 

\begin{proposition}[see e.g.~{\cite[Theorem 5.9]{Conway}}]
Let $G$ be a region in ${\mathbb{C}}$ and let $\{f_n\}$ be a sequence in $H(G)$ such that no $f_n$ is identically zero. 
If $\sum_{n=1}^\infty (f_n(s)-1)$ converges absolutely and uniformly on any compact subset of $G$ then $\prod_{n=1}^\infty f_n(s)$ converges to an analytic function $f(s)$ in $H(G)$. 
If $z$ is a zero of $f$ then $z$ is a zero of only a finite number of the functions $f_n$, and the multiplicity of the zero of $f$ at $z$ is the sum of the multiplicities of the zeros of the functions $f_n$ at $z$. 
\label{pro:co5.9}
\end{proposition}

\begin{proof}[Proof of Theorem \ref{th:EPc}]
Put $v := \min_{1\le l\le m}\Re \langle \va_l,\vs\rangle$.
Then, by the assumption $v>1$ and $-1 \le \alpha_l(p) \le 1$ for any $p\in\Prime$ and $1\le l\le m$, we have
$$
\sum_{p} \bigl| \alpha_l(p) p^{-\langle \va_l,\vs\rangle} \bigr| \le \sum_{p} p^{-v} \le \sum_{n\ge 2} n^{-v}
\le \int_1^\infty x^{-v} dx < \infty .
$$
Thus $\sum_{p} \alpha_l(p) p^{-\langle \va_l,\vs\rangle}$ converges absolutely and uniformly on any compact subset of the region $\min_{1\le l\le m}\Re \langle \va_l,\vs\rangle$$ >1$. 
By Proposition \ref{pro:co5.9}, the product \eqref{eq:def1} converges absolutely in the region $\min_{1\le l \le m}\Re \langle \va_l,\vs\rangle >1$. 

We also have that $0<|1-\alpha_l(p) p^{-\langle \va_l,\vs\rangle} |^{-1}$ for any $\vs\in\mathbb{C}^d$, $p\in\Prime$ and $1\le l\le m$ , so that $\eqref{eq:def1}$ does not have zeros.
\end{proof}

Here and in the sequel, we define $\log Z_E (\vs)$ by the following Dirichlet series expansion
\begin{equation}\label{eq:9.19}
\log Z_E (\vs) := \sum_p \sum_{r=1}^\infty \sum_{l=1}^m \frac{1}{r} \alpha_l(p)^r p^{-r \langle \va_l,\vs\rangle}
\end{equation}
in the region of absolute convergence (see e.g.~\cite[(9.19)]{Steuding1}). This formula plays important role in the proof of Main theorem. 

\subsection{Examples of multidimensional polynomial Euler products}

Some simple examples of $Z_E(\vs)$ for $d=1$ are the following.

\begin{example}\label{ex:EP}
$(i)$ When $d=m=1$, $a=1$ and $\al(p)= -1$, then
$$
Z_E (s_1) = \prod_p \frac{1}{1 +p^{-s_1}} = \prod_p \frac{1 -p^{-s_1}}{1 -p^{-2s_1}}  = 
\frac{\zeta (2s_1)}{\zeta (s_1)}.
$$

\pn
$(ii)$ When $d=m=1$, $a=2$, $\al(p)=1$, or $d=1$, $m=2$, $a_1=a_2=1$, and $\al_1(p) = - \al_2(p) =1$, then
$$
Z_E (s_1) = \prod_p \frac{1}{1 -p^{-2s_1}}  = \zeta (2s_1) = \prod_p \frac{1}{(1 -p^{-s_1})(1 +p^{-s_1})}.
$$
\end{example}

Similarly, we have following examples for $d=2$ as a simple multidimensional case.

\begin{example}
\pn
$(iii)$ When $d=m=2$, $\va_1=(1,0)$, $\va_2=(1,1)$ and $\al_l(p)=1$, $l=1,2$, then
$$
Z_E (\vs)=\prod_p \frac{1}{1 -p^{-s_1}} \frac{1}{1 -p^{-(s_1+s_2)}}=\zeta (s_1)\zeta (s_1+s_2). 
$$
\pn
$(iv)$ When $d=m=2$, $\va_1=(1,0)$, $\va_2=(1,2)$, $\al_1(p)=1$ and $\al_2(p)=\chi(p)$, where $\chi (p)$ is a real Dirichlet character, then
$$
Z_E (\vs)=\prod_p \frac{1}{1 -p^{-s_1}} \frac{1}{1 -\chi(p)p^{-(s_1+2s_2)}}=\zeta(s_1)L(s_1+2s_2,\chi).
$$
\end{example}

We give further multidimensional examples and mention their behaviors related to probability theory in Section 3.

\section{Multidimensional compound Poisson zeta distributions}

In this section, we define multidimensional compound Poisson zeta distributions on $\rd$ generated by the multidimensional polynomial Euler products defined in the previous section. We put
$$
\vs:=\vsig +{\rm i}\vt, \q\Vec{\sigma}, \vt\in\rd,
$$
and, for $\vsig$ satisfying $\min_{1\le l\le m}\Re \langle \va_l,\vs\rangle >1$, a normalized function
$$
f_{\vsig}\left(\vt\,\right):=\frac{Z_E\left(\vsig +{\rm i}\vt\,\right)}{Z_E(\vsig)}.
$$ 

\subsection{Multidimensional \textit{m}-tuple compound Poisson zeta distribution}

We have the following.

\begin{theorem}\label{th:cfep}
Let $\va_{1}= \cdots = \va_{m} =:\va$ and $\alpha_l(p)=0$ or \,$\pm1$ for $1\le l\le m$ and $p\in\Prime$ in \eqref{eq:def1}.
Then $f_{\vsig}$ is a characteristic function if and only if $\sum_{l=1}^m \alpha_l(p) \ge 0$ for all $p\in\Prime$.
Moreover, when $\sum_{l=1}^m \alpha_l(p) \ge 0$ for all $p\in\Prime$, $f_{\vsig}$ is a compound Poisson characteristic function with its finite L\'evy measure $N^{(\va)}_{\vsig}$ on $\rd$ given by 
\begin{align}\label{eq:lm1}
N^{(\va)}_{\vsig} (dx) = \sum_p \sum_{r=1}^{\infty} \sum_{l=1}^m \frac{1}{r}
\alpha_l(p)^r p^{-r\langle\va,\vsig\rangle} \delta_{\log p^r \va} (dx).
\end{align}
\end{theorem}

For the proof of Theorem \ref{th:cfep}, we use the following two lemmas.

\begin{lemma}
Suppose $\alpha_{l}= 0,\pm 1$ for $1\le l\le m$. 
Then $\sum_{l=1}^m \alpha_{l}^r \ge 0$ for any $r \ge 2$ if and only if $\sum_{l=1}^m \alpha_{l}\ge 0$. 
\end{lemma}
\begin{proof}
It is obvious that we have $\sum_{l=1}^m \alpha_{l}^r \ge 0$ for any $r \ge 2$ if $\sum_{l=1}^m \alpha_{l}\ge 0$. 

Conversely, let $k\in\N$ and suppose $\sum_{l=1}^m \alpha_{l}^r \ge 0$ for any $r \ge 2$. 
When $r=2k+1$, we have $\sum_{l=1}^m \alpha_l(p)=\sum_{l=1}^m \alpha_{l}^{2k+1}\ge 0$ since $\alpha_{l}=\alpha_{l}^{2k+1}$.
Thus we have $\sum_{l=1}^m \alpha_{l}\ge 0$ if $\sum_{l=1}^m \alpha_{l}^r \ge 0$ for any $r \ge 2$.
\end{proof}

\begin{remark}
If $m=2$, we can remove $|\alpha_{l}|=1$ since $|\alpha_1| > |\alpha_2|$ if and only if $|\alpha_1|^r > |\alpha_2|^r$. 
We can not do so when $m \ge 3$. 
As a counter example, we have the case $\alpha_1=\alpha_2=1/3$ and $\alpha_3=-2/3$.
\end{remark}

\begin{lemma}\label{lm:q}
If there exists a set of prime numbers $q$ satisfying $\sum_{l=1}^m \alpha_l(q) < 0$, then there exists $\vt_0\in\rd$ such that $|f_{\vsig}(\vt_0)|>1$. 
\label{lem:key}
\end{lemma}
For the proof of this lemma, we use linear independence of real numbers and the Kronecker's approximation theorem.

It is called that real numbers $\theta_1,\ldots,\theta_n$ are {\it linearly independent over the rationals} if $\sum_{k=1}^n c_k \theta_k=0$ with rational multipliers $c_1,\ldots,c_n$ implies $c_1 = \cdots = c_n =0$. 
Put
\begin{equation}
\theta_k := \frac{\log p_k}{2\pi}, \qq 1 \le k \le n,
\label{eq:deftheta}
\end{equation}
where $p_1, \ldots, p_n$ are the first $n$ primes. 
Then $\theta_k$ are linearly independent over the rationals and it can be shown by this way. 
When we suppose $\sum_{k=1}^n c_k \theta_k =0$, namely, $\log (p_1^{c_1} \cdots p_n^{c_n}) =0$, it implies that $p_1^{c_1} \cdots p_n^{c_n}=1$. 
Hence we obtain $c_1 = \cdots = c_n =0$ by the fundamental theorem of arithmetic (or the unique-prime-factorization theorem). 

The following proposition is called the (first form of) Kronecker's approximation theorem. 
\begin{proposition}[{\cite[Theorem 7.9]{Apo2}}]\label{pro:kroap1}
If $\phi_1 , \ldots ,\phi_n$ are arbitrary real numbers, if real numbers $\theta_1,\ldots,\theta_n$ are linearly independent over the rationals, and if $\varepsilon>0$ is arbitrary, then there exists a real number $t$ and integers $h_1, \ldots, h_n$ such that
$$
|t\theta_k - h_k - \phi_k| < \varepsilon, \qq 1 \le k \le n. 
$$
\end{proposition}

\begin{proof}[Proof of Lemma \ref{lem:key}]
Recall that $\log Z_E (\vs)$ is defined by (\ref{eq:9.19}). 
Denote by $\Prime^-$ the set of prime numbers $q$ satisfying $\sum_{l=1}^m \alpha_l(q) < 0$ and $\Prime^+:=\Prime\backslash\Prime^-$, and define a vector-valued function $D(\vt)$, $\vt\in\rd$, as follows:
\begin{equation*}
\begin{split}
D(\vt):=& \log \left| \frac{Z_E (\vsig + {\rm{i}}\vt)}{Z_E (\vsig)} \right| = 
\frac{1}{2} \log \frac{Z_E (\vsig + {\rm{i}}\vt)}{Z_E (\vsig)} \frac{Z_E (\vsig - {\rm{i}}\vt)}{Z_E (\vsig)}\\
= &\frac{1}{2} \sum_p \sum_{r=1}^{\infty} \sum_{l=1}^m \frac{1}{r} 
\left( \alpha_l(p)^r p^{-r\langle\va,\vsig\rangle} 
\bigl(p^{r\langle\va,{\rm{i}}\vt\rangle} + p^{-r\langle\va,{\rm{i}}\vt\rangle} -2 \bigr)\right) .
\end{split}
\end{equation*}
Let $K\in\N$, $\sum_{r,p>2K}'$ be a sum taken over $r>2K$ or $p>2K$, $\Prime^+_K := \{ p\in\Prime^+ : 2 \le p \le 2K \}$ and $\Prime^-_K := \{ q\in\Prime^- : 2 \le q \le 2K \}$. 
For any $\ep >0$, we can see that there exists an integer $K$ such that $|2\sum_{r,p>2K}' \sum_{l=1}^m r^{-1} \alpha_l(p)^rp^{-r\langle\va,\vsig\rangle}|< \varepsilon$ and $\Prime^-_K \ne \emptyset$ by the absolute convergence of $Z_E(\vs)$ (see, Theorem \ref{th:EPc}).
In the view of $p_k^{{\rm{i}}t}= e^{{\rm{i}}t\log p_k} = e^{2\pi {\rm{i}}t\theta_k}$, \eqref{eq:deftheta} and by Proposition \ref{pro:kroap1}, for any $\ep '>0$ independent of $\ep$ and $K$, there exists $T_0:= \langle\va,\vt_0\rangle$, $\vt_0\in\rd$, such that
$$
|p^{{\rm{i}}T_0}-1| < \varepsilon', \q p \in \Prime^+_K, \q \mbox{and} \q 
|q^{{\rm{i}}T_0}+1| < \varepsilon', \q q \in \Prime^-_K.
$$
By the factorization $x^r-1 = (x-1)(x^{r-1}+\cdots+1)$, for any $1 \le r \le 2K$ and $p\in\Prime^+_K$, we have $|p^{r{\rm{i}}T_0}-1| < r\varepsilon' \le 2K\varepsilon'$. 
Similarly, by the factorization $x^{2k-1}+1 = (x+1)(x^{2k-2}-x^{2k-3}+\cdots+1)$ when $r=2k-1\in 2\N -1$ with $k\le K$, one has $|q^{(2k-1){\rm{i}}T_0}+1| < (2k-1)\varepsilon' \le 2K\varepsilon'$ for any $q\in\Prime^-_K$. 
Also, by the factorization $x^{2k}-1 = (x+1)(x-1)(x^{2k-2}+x^{2k-4}+\cdots+1)$ when $r=2k\in 2\N$ with $k\le K$, it holds that $|q^{2k{\rm{i}}T_0}-1| < 2k \varepsilon' \le 2K\varepsilon'$ for any $q\in\Prime^-_K$.
Hence there exits $T_0\in\R$ such that
\begin{equation*}
\begin{split}
-4K\varepsilon' < p^{r{\rm{i}}T_0} + p^{-r{\rm{i}}T_0} -2 \le 0, \q
&p\in\Prime^+_K, \,\, 1 \le r \le 2K ,\\
-4 -4K\varepsilon' < q^{(2k-1){\rm{i}}T_0} + q^{-(2k-1){\rm{i}}T_0} - 2 < -4 + 4K\varepsilon' , \q
&q\in\Prime^-_K, \,\, 1 \le k \le K,\\ 
-4K\varepsilon' \le q^{2k{\rm{i}}T_0} + q^{-2k{\rm{i}}T_0} - 2 \le 0, \q
&q\in\Prime^-_K, \,\, 1 \le k \le K.
\end{split}
\end{equation*}
By $\sum_{l=1}^m \alpha_l(p) \ge 0$, $\sum_{l=1}^m \alpha_l(q)^{2k-1} <0$ and $\sum_{l=1}^m \alpha_l(q)^{2k} >0$, for any  $p \in\Prime^+_K$ and $q\in\Prime^-_K$, we have
\begin{align}
C_1 := \sum_{p \in \Prime^+_K} \sum_{r=1}^{2K} \sum_{l=1}^m 
\frac{1}{r} \alpha_l(p)^r p^{-r\langle\va,\vsig\rangle}
&\ge 0 , \label{eq:lem34;1}\\
C_2 := \sum_{q \in \Prime^-_K}\sum_{k=1}^{K} \sum_{l=1}^m 
\frac{1}{2k} \alpha_l(q)^{2k} q^{-2k\langle\va,\vsig\rangle} 
&\ge 0 , \label{eq:lem34;2}\\
C_3 := \sum_{q \in \Prime^-_K}\sum_{k=1}^{K} \sum_{l=1}^m \frac{1}{2k-1} \alpha_l(q)^{2k-1} q^{-(2k-1)\langle\va,\vsig\rangle} &< 0 \label{eq:lem34;3}.
\end{align}
Note that if there do not exist any $p \in \Prime^+_K$, we regard the summation in (\ref{eq:lem34;1}) is $0$. Therefore we obtain
\begin{equation*}
\begin{split}
D(\vt_0) >&-\ep -2K\varepsilon' \sum_{p \in \Prime^+_K} \sum_{r=1}^{2K} \sum_{l=1}^m 
\frac{1}{r} \alpha_l(p)^r p^{-r\langle\va,\vsig\rangle} -2K\varepsilon' \sum_{q \in \Prime^-_K}\sum_{k=1}^{K} \sum_{l=1}^m \frac{1}{2k} \alpha_l(q)^{2k} q^{-2k\langle\va,\vsig\rangle} \\
& + (2K\varepsilon' -2)\sum_{q \in \Prime^-_K}\sum_{k=1}^{K} \sum_{l=1}^m 
\frac{1}{2k-1} \alpha_l(q)^{2k-1} q^{-(2k-1)\langle\va,\vsig\rangle} \\
=& -\ep -2K\varepsilon' (C_1 + C_2 -C_3) -2 C_3.
\end{split}
\end{equation*}
Suppose $\varepsilon$ is sufficiently small and $\varepsilon'$ such that $KC'\varepsilon' < \varepsilon$.
Then we have $D(\vt_0) >0$ by (\ref{eq:lem34;1}), (\ref{eq:lem34;2}) and (\ref{eq:lem34;3}). This completes the proof.
\end{proof}

\begin{proof}[Proof of Theorem \ref{th:cfep}]
By Proposition \ref{pro:Sato} $(i)$ and Lemma \ref{lm:q}, if there exists a set of prime numbers $q$ satisfying $\sum_{l=1}^m\alpha_l(q)<0$ then $f_{\vsig}$ is not a characteristic function.
Thus we only have to show that $f_{\vsig}$ is a compound Poisson characteristic function with a finite L\'evy measure $N^{(\va)}_{\vsig}$ on $\rd$ given in $\eqref{eq:lm1}$ if $\sum_{l=1}^m \alpha_l(p) \ge 0$ for all $p\in\Prime$.  
It is easy to see that $N^{(\va)}_{\vsig}$ is a measure on $\rd$ since $N^{(\va)}_{\vsig}(x)\ge 0$ for every $x\in\rd$ when $\sum_{l=1}^m \alpha_l(p) \ge 0$ for all $p\in\Prime$.
Now put $v := \langle \va,\vsig\rangle>1$. 
Then $\zeta (v)$ is a positive constant (see, Section 1.2) and note that $\alpha_l(p)=0,\pm1$ for $1\le l\le m$ and $p\in\Prime$.
So that we have
\begin{equation*}
\begin{split}
N^{(\va)}_{\vsig} (\rd) =& \int_{\rd}\sum_p \sum_{r=1}^{\infty} \sum_{l=1}^m \frac{1}{r}
\alpha_l(p)^r p^{-r\langle\va,\vsig\rangle} \delta_{\log p^r \va} (dx)
= \sum_p \sum_{r=1}^{\infty} \sum_{l=1}^m \frac{1}{r} \alpha_l(p)^r p^{-r\langle\va,\vsig\rangle} \\ 
\le &\sum_p \sum_{r=1}^{\infty} \sum_{l=1}^m \frac{1}{r} p^{-r\langle\va,\vsig\rangle} \le 
m\sum_p \sum_{r=1}^{\infty} p^{-rv}
 \le m \sum_{n=2}^{\infty} \sum_{r=1}^{\infty} n^{-rv} \\
=& \, m\sum_{n=2}^{\infty} \frac{n^{-v}}{1-n^{-v}}
\le 2m \sum_{n=2}^{\infty} n^{-v} = 2m \left(\zeta (v) -1\right)<\infty.
\end{split}
\end{equation*}
Thus, $N^{(\va)}_{\vsig}$ is a finite measure on $\rd$.

By Theorem \ref{th:EPc}, for $\vt\in\rd$, the function $f_{\vsig}(\vt)$ converges when $\langle \va,\vsig\rangle >1$.
Then, we have 
\begin{align}
\log f_{\vsig}(\vt)=& \log \frac{Z_E (\vsig + {\rm i}\vt)}{Z_E (\vsig)} 
= \sum_p \sum_{r=1}^{\infty} \sum_{l=1}^m \frac{1}{r}\alpha_l(p)^r p^{-r\langle\va,\vsig\rangle} 
\bigl(p^{-r\langle\va,{\rm i}\vt\rangle} -1\bigr)\nonumber\\ 
= & \sum_p \sum_{r=1}^{\infty} \sum_{l=1}^m \frac{1}{r} 
\alpha_l(p)^r p^{-r\langle\va,\vsig\rangle} \bigl(e^{-r\langle\va,{\rm i}\vt\rangle \log p} -1\bigr)\nonumber\\ 
= &\int_{\rd} (e^{-\langle{\rm i}\vt,x\rangle}-1) \sum_p \sum_{r=1}^{\infty} \sum_{l=1}^m \frac{1}{r}
\alpha_l(p)^r p^{-r\langle\va,\vsig\rangle} \delta_{\log p^r \va} (dx)\nonumber\\
=&\int_{\rd} (e^{-\langle{\rm i}\vt,x\rangle}-1)N^{(\va)}_{\vsig}(dx).\label{eq:chf1}
\end{align}

It is also easy to see that the measure $N^{(\va)}_{\vsig}$ satisfies $\int_{|x|<1}|x|N^{(\va)}_{\vsig}(dx)\le N^{(\va)}_{\vsig}(\rd)<\infty$.
Thus $f_{\vsig}$ is an infinitely divisible characteristic function and $N^{(\va)}_{\vsig}$ is the L\'evy measure of $f_{\vsig}$ by Propositions \ref{pro:LK1} and \ref{pro:LK2}.
Regarding $c=N^{(\va)}_{\vsig}(\rd)$ and $\rho (dx) =c^{-1}N^{(\va)}_{\vsig}(dx)$ in \eqref{CPcf} of Example \ref{examples}, we can see that $\eqref{eq:chf1}$ is of the form of a compound Poisson characteristic function. 
Hence $f_{\vsig}$ is a compound Poisson characteristic function with a finite L\'evy measure $N^{(\va)}_{\vsig}$ on $\rd$. 
This completes the proof.

\end{proof}

Now we define a multidimensional zeta distribution generated by $\mZ_E$.

\begin{definition}[Multidimensional \textit{m}-tuple compound Poisson zeta distribution]
Let $\va_{1}= \cdots = \va_{m} =:\va$ and $\alpha_l(p)=0,\pm1$, for $1\le l\le m$ and $p\in\Prime$, with $\sum_{l=1}^m \alpha_l(p) \ge 0$ for all $p\in\Prime$ in \eqref{eq:def1}.
A distribution on $\rd$ is a multidimensional \textit{m}-tuple compound Poisson zeta distribution if it has a characteristic function 
$$
f_{\vsig}\left(\vt\,\right) =\frac{Z_E\left(\vsig +{\rm i}\vt\,\right)}{Z_E(\vsig)}.
$$
\end{definition}

As to give some examples of multidimensional \textit{m}-tuple compound Poisson zeta distribution, we use following functions.
Let $\Re (s) >1$ and $\chi (n)$ be a real non-principal Dirichlet character. 
Put
\begin{equation*}
\begin{split}
&L_1 (s) = \prod_p ( 1 - \beta(p) p^{-s})^{-1} , \qquad \beta(p) := 
\begin{cases}
-1 & p=2 , \\
1 & \mbox{otherwise}, 
\end{cases}\\
&L_2 (s) = \prod_p ( 1 - \gamma(p) p^{-s})^{-1} , \qquad \gamma(p) := 
\begin{cases}
-1 & p=3 , \\
1 & \mbox{otherwise}, 
\end{cases}
\end{split}
\end{equation*}

Then we have following examples.

\begin{example}\label{ex:41}
\pn
$(i)$ Functions to be \textit{m}-tuple compound Poisson zeta.
$$
\zeta (s), \q \zeta (s) L(s,\chi), \q L_1 (s) L_2(s). 
$$
\pn
$(ii)$ Functions not to generate probability distributions. 
$$
L(s,\chi), \q L_1 (s), \q L_2(s).
$$
\end{example}

\subsection{Multidimensional \textit{m}-rank compound Poisson zeta distribution}

For $\va\in\rd$, we call that $\rd$-valued vectors $\va_1, \dots ,\va_m$ are {\it linearly dependent but linearly independent over the rationals} if
$\va_l =\psi_l\va$, $1\le l\le m$, where $\psi_l$ are algebraic real numbers and linearly independent over the rationals.
Denote by LI and LR, the conditions of $\rd$-valued vectors $\va_1, \dots ,\va_m$, where\\ 
(LI) linearly independent,\\
(LR) linearly dependent but linearly independent over the rationals, \\
respectively.

\begin{theorem}\label{th:cfep3}
Suppose that $\R^d$-valued vectors $\va_{1}, \ldots , \va_{m} $ satisfy the condition LI or LR in \eqref{eq:def1}. 
Then $f_{\vsig}$ is a characteristic function if and only if $\alpha_l(p) \ge 0$ for all $1\le l\le m$ and $p\in\Prime$.
Moreover, $f_{\vsig}$ is a compound Poisson characteristic function with its finite L\'evy measure $N_{\vsig}$ on $\rd$ given by 
\begin{equation}\label{eq:lm2}
N_{\vsig} (dx) = \sum_p \sum_{r=1}^{\infty} \sum_{l=1}^m \frac{1}{r}
\alpha_l(p)^r p^{-r\langle\va_l,\vsig\rangle} \delta_{\log p^r \va_l} (dx).
\end{equation}
\end{theorem}

For the proof of Theorem \ref{th:cfep3}, we need following two propositions and one lemma.  
The next proposition is given by Baker.
\begin{proposition}[{\cite[Theorem 2.4]{Baker}}]
The numbers $\gamma_1^{\beta_1} \cdots \gamma_n ^{\beta_n}$ are transcendental for any algebraic numbers $\gamma_1 , \ldots , \gamma_n$, other than $0$ or $1$, and any algebraic numbers $\beta_1, \ldots , \beta_n$ with $1,\beta_1, \ldots , \beta_n$ are linearly independent over the rationals.
\label{lem:baker}
\end{proposition}
By using this fact, we have the following. 
\begin{proposition}[{\cite[Proposition 2.2]{Nakamurasr1}}]
Let $p_n$ be the $n$-th prime number and $\omega_1, \omega_2, \ldots ,$ $\omega_m$ with $\om_1 =1$ be algebraic real numbers which are linearly independent over the rationals. 
Then $\{ \log p_n ^{\omega_l}\}_{n \in {\mathbb{N}}}^{1 \le l \le m}$ is also linearly independent over the rationals. 
\label{pro:abaker}
\end{proposition}

For the convenience of readers, we write its proof here. 

\begin{proof}
Denote by $\mathbb{Q}$ the set of all rational numbers.
Suppose 
$$
\sum_{n=1}^r c_{1n} \log p_n + \sum_{n=1}^r c_{2n} \log p_n^{\omega_2} + \cdots + \sum_{n=1}^r c_{mn} \log p_n ^{\omega_m} = 0 , 
$$
for  $r,n\in\N$, $1\le l\le m$ and $c_{ln} \in {\mathbb{Q}}$.
By the formula above, we have
\begin{equation}
p_1^{c_{11}} \cdots p_r^{c_{1r}} = ( p_1^{c_{21}} \cdots p_r^{c_{2r}} )^{-\omega_2} \cdots ( p_1^{c_{m1}} \cdots p_r^{c_{mr}} )^{-\omega_m} .
\label{eq:mu}
\end{equation}
The left-hand side of (\ref{eq:mu}) is an algebraic number. 
But the right-hand side of (\ref{eq:mu}) is transcendental when $( c_{21} , \ldots , c_{2r}, \ldots , c_{m1} , \ldots , c_{mr} ) \ne ( 0, \ldots , 0)$ by Lemma \ref{lem:baker} and the unique factorization of prime numbers. 
If $c_{h1} = \cdots = c_{hr} =0$, for some $h$, $2 \le h \le m$, we can apply a lower dimensional case of Baker's theorem. 
Hence we have $( c_{21} , \ldots , c_{2r}, \ldots , c_{m1} , \ldots , c_{mr} ) = ( 0, \ldots , 0)$. 
Moreover, we obtain $c_{11} = \cdots = c_{1r} =0$ by the unique factorization of prime numbers.
\end{proof}

Now we give a following lemma.

\begin{lemma}\label{lm:ln}
If there exists a set of pairs of $h$, $1\le h\le m$, and prime numbers $q$ satisfying $\alpha_l(q) < 0$, then there exists $\vt_0\in\rd$ such that $|f_{\vsig}(\vt_0)|>1$. 
\end{lemma}

\begin{proof}
Recall that $\log Z_E (\vs)$ is given by (\ref{eq:9.19}). 
Denote by $A^-$ and $A^+$ the set of pairs of numbers $\{h,q\}$ and $\{l,p\}$, where $1\le l\le m$ and $1\le h\le m$, and prime numbers $p$ and $q$ satisfying $\alpha_h(q) < 0$ and $\alpha_l(p)\ge0$, respectively. 
Moreover, let $K \in \N$, $A^+_K := \{ \{l,p\} \in A^+ : 2 \le p \le 2K\}$ and $A^-_K := \{ \{h,q\} \in A^- : 2 \le q \le 2K\}$. 
First we consider the case when $\va_{1}, \ldots , \va_{m} $ satisfy the condition LI.
Let $\om_1,\omega_2, \ldots ,\omega_m$ with $\omega_1=1$ be algebraic real numbers which are linearly independent over the rationals. 
Then there exits $\vt_0\in\rd$ such that $(\langle \va_1, \vt_0 \rangle, \ldots , \langle \va_m, \vt_0 \rangle) =(\omega_1, \omega_2, \ldots ,\omega_{m})$ since $\va_{1}, \ldots , \va_{m} $ are LI.
When $\va_{1}, \ldots , \va_{m} $ satisfy the condition LR, put $\om_l:=\psi_l(\psi_1\langle \va,\vt_0\rangle)^{-1}$, for $1\le l\le m$ and $\vt_0\in\rd\backslash \{0\}$. Then $\om_1,\omega_2, \ldots ,\omega_m$ are algebraic real numbers having linear independence over the rationals with $\om_1=1$.
In both cases, we can say the following.
Define a vector-valued function $D(T\vt_0)$, $T\in\R$, as follows:
\begin{equation*}
\begin{split}
D(T\vt_0)&:= \log \left| \frac{Z_E (\vsig + {\rm{i}}T\vt_0)}{Z_E (\vsig)} \right| \\
&=\frac{1}{2} \sum_p \sum_{r=1}^{\infty} \sum_{l=1}^m \frac{1}{r} 
\left( \alpha_l(p)^r p^{-r\langle\va_l,\vsig\rangle} 
\bigl(p^{r{\rm{i}}T\langle\va_l,\vt_0\rangle} + p^{-r{\rm{i}}T\langle\va_l,\vt_0\rangle} -2 \bigr)\right) \\
&=\frac{1}{2} \sum_p \sum_{r=1}^{\infty} \sum_{l=1}^m \frac{1}{r} 
\Bigl( \alpha_l(p)^r p^{-r\langle\va_l,\vsig\rangle} 
\bigl(p^{r{\rm{i}}T\omega_l} + p^{-r{\rm{i}}T\omega_l} -2 \bigr)\Bigr) .
\end{split}
\end{equation*} 
Then, for any $\ep >0$, we can see that there exists an integer $K$ such that 
$|2\sum_{r,p>2K}' \sum_{l=1}^m$ $r^{-1}\alpha_l(p)^rp^{-r\langle\va_l,\vsig\rangle}|< \varepsilon$ and 
$A^-_K \neq\emptyset$ by the absolute convergence of $Z_E(\vs)$ (see, Theorem \ref{th:EPc}).
In the view of $p_k^{{\rm{i}}\omega_lt}= e^{{\rm{i}}t\omega_l\log p_k}$ and Propositions \ref{pro:kroap1} and \ref{pro:abaker}, 
for any $\ep'>0$ independent of $\ep$ and $K$, there exists $T_0 \in \R$, such that
\begin{equation*}
|p^{{\rm{i}}\om_lT_0}-1| < \varepsilon', \q \{l,p\} \in A^+_K \q \mbox{and} \q 
|q^{{\rm{i}}\om_{h}T_0}+1| < \varepsilon', \q \{h,q\} \in A^-_K.
\end{equation*}
By the factorization $x^r-1 = (x-1)(x^{r-1}+\cdots+1)$, for any $1 \le r \le 2K$ and $\{l,p\}\in A^+_K$, we have 
$|p^{r{\rm{i}}\omega_lT_0}-1| < r\varepsilon' \le 2K\varepsilon'$. 
Similarly, by the factorization $x^{2k-1}+1 = (x+1)(x^{2k-2}-x^{2k-3}+\cdots+1)$, when $r=2k-1\in 2\N -1$ with $k\le K$, one has $|q^{(2k-1){\rm{i}}\om_{h}T_0}+1| < (2k-1) \varepsilon' \le 2K\varepsilon'$ for any $\{h,q\}\in A^-_K$.
Also, by the factorization $x^{2k}-1 = (x+1)(x-1)(x^{2k-2}+x^{2k-4}+\cdots+1)$ when $r=2k\in 2\N$ with $k\le K$, it holds that 
$|q^{2k{\rm{i}}\om_{h}T_0}-1| < 2k \varepsilon' \le 2K\varepsilon'$ for any $\{h,q\}\in A^-_K$.
Hence there exits $T_0\in\R$ such that
\begin{equation*}
\begin{split}
-4K\varepsilon' < p^{r{\rm{i}}\omega_lT_0} + p^{-r{\rm{i}}\omega_lT_0} -2 \le 0,
\q &\{l,p\}\in A^+_K, \,\, 1\le r\le 2K, \\
-4 -4K\varepsilon' < q^{(2k-1){\rm{i}}\om_{h}T_0} + q^{-(2k-1){\rm{i}}\om_{h}T_0} - 2 < -4 + 4K\varepsilon' , \q
&\{h,q\}\in A^-_K,\,\, 1\le k\le K, \\
-4K\varepsilon' \le q^{2k{\rm{i}}\om_{h}T_0} + q^{-2k{\rm{i}}\om_{h}T_0} - 2 \le 0, \q &\{h,q\}\in A^-_K,\,\,1\le k\le K.
\end{split}
\end{equation*}
By $\alpha_l(p) \ge 0$ and $\alpha_h(q) < 0$ for any $\{l,p\}\in A^+_K$ and $\{h,q\}\in A^-_K$, respectively, we have
\begin{align}
\sum_{\{l,p\}\in A^+_K} \sum_{r=1}^{2K} \frac{1}{r} \alpha_l(p)^r p^{-r\langle\va_l,\vsig\rangle}, \,\,\,
\sum_{\{h,q\}\in A^-_K}\sum_{k=1}^{K} \frac{1}{2k} \alpha_h(q)^{2k} q^{-2k\langle\va_h,\vsig\rangle}&\ge 0,
\label{eq:lem311;1} \\
\sum_{\{h,q\}\in A^-_K}\sum_{k=1}^{K} \frac{1}{2k-1} \alpha_h(q)^{2k-1} q^{-(2k-1)\langle\va_h,\vsig\rangle} &< 0
\label{eq:lem311;2}.
\end{align}
Note that if there do not exist any $p \in A^+_K$, we regard the summation $\sum_{\{l,p\}\in A^+_K}$ in (\ref{eq:lem311;1}) is $0$. Therefore we obtain 
\begin{equation*}
\begin{split}
D(T_0\vt_0) >& -\ep-2K\varepsilon'
\sum_{\{l,p\}\in A^+_K} \sum_{r=1}^{2K} \frac{1}{r} \alpha_l(p)^r p^{-r\langle\va_l,\vsig\rangle}
-2K\varepsilon' \sum_{\{h,q\}\in A^-_K}\sum_{k=1}^{K} \frac{1}{2k} 
\alpha_h(q)^{2k} q^{-2k\langle\va_h,\vsig\rangle} \\
&+ (2K\varepsilon' -2) \sum_{\{h,q\}\in A^-_K}\sum_{k=1}^{K} \frac{1}{2k-1} \alpha_h(q)^{2k-1} q^{-(2k-1)\langle\va_h,\vsig\rangle} .
\end{split}
\end{equation*}
Now suppose $\varepsilon$ is sufficiently small and $\varepsilon'$ such that $K\varepsilon' < \varepsilon$. 
Then one has $D(T_0\vt_0) >0$ from (\ref{eq:lem311;1}) and (\ref{eq:lem311;2}). This completes the proof.
\end{proof}

\begin{proof}[Proof of Theorem \ref{th:cfep3}]
By Proposition \ref{pro:Sato} $(i)$ and Lemma \ref{lm:ln}, if there exists a set of $h$, $1\le h\le m$, and prime numbers $q$ satisfying $\alpha_h(q) < 0$ then $f_{\vsig}$ is not a characteristic function.
Thus we only have to show that $f_{\vsig}$ is a compound Poisson characteristic function with a finite L\'evy measure $N_{\vsig}$ on $\rd$ given in $\eqref{eq:lm2}$ if $\alpha_l(p) \ge 0$ for all $1\le l\le m$ and $p\in\Prime$.  
It is easy to see that $N_{\vsig}$ is a measure on $\rd$ since $N_{\vsig}(x)\ge 0$ for every $x\in\rd$ when $\alpha_l(p) \ge 0$ for all $1\le l\le m$ and $p\in\Prime$.
Now put $v := \min_{1\le l\le m}\langle \va_l,\vsig\rangle>1$. 
Then $\zeta (v)$ is a positive constant (see, Section 1.2) and note that $0\le \alpha_l(p)\le 1$ for $1\le l\le m$ and $p\in\Prime$.
So that we have
\begin{equation*}
\begin{split}
N_{\vsig} (\rd) =& \int_{\rd}\sum_p \sum_{r=1}^{\infty} \sum_{l=1}^m \frac{1}{r}
\alpha_l(p)^r p^{-r\langle\va_l,\vsig\rangle} \delta_{\log p^r \va_l} (dx)
= \sum_p \sum_{r=1}^{\infty} \sum_{l=1}^m \frac{1}{r} \alpha_l(p)^r p^{-r\langle\va_l,\vsig\rangle} \\ 
\le &\sum_p \sum_{r=1}^{\infty} \sum_{l=1}^m \frac{1}{r} p^{-r\langle\va_l,\vsig\rangle} \le 
m\sum_p \sum_{r=1}^{\infty} p^{-rv}
 \le m \sum_{n=2}^{\infty} \sum_{r=1}^{\infty} n^{-rv} \\
=& \, m\sum_{n=2}^{\infty} \frac{n^{-v}}{1-n^{-v}}
\le 2m \sum_{n=2}^{\infty} n^{-v} = 2m \left(\zeta (v) -1\right)<\infty.
\end{split}
\end{equation*}
Thus, $N_{\vsig}$ is a finite measure on $\rd$.

By Theorem \ref{th:EPc}, for $\vt\in\rd$, the function $f_{\vsig}(\vt)$ converges when $\min_{1\le l\le m}\langle \va_l,\vsig\rangle >1$.
Then, we have 
\begin{align}
\log f_{\vsig}(\vt)=& \log \frac{Z_E (\vsig + {\rm i}\vt)}{Z_E (\vsig)} 
= \sum_p \sum_{r=1}^{\infty} \sum_{l=1}^m 
\frac{1}{r}\alpha_l(p)^r p^{-r\langle\va_l,\vsig\rangle} \bigl(p^{-r\langle\va_l,{\rm i}\vt\rangle} -1\bigr)
\nonumber\\ 
=& \sum_p \sum_{r=1}^{\infty} \sum_{l=1}^m \frac{1}{r} 
\alpha_l(p)^r p^{-r\langle\va_l,\vsig\rangle} \bigl(e^{-r\langle\va_l,{\rm i}\vt\rangle \log p} -1\bigr)\nonumber\\ 
= &\int_{\rd} (e^{-\langle{\rm i}\vt,x\rangle}-1) \sum_p \sum_{r=1}^{\infty} \sum_{l=1}^m \frac{1}{r}
\alpha_l(p)^r p^{-r\langle\va_l,\vsig\rangle} \delta_{\log p^r \va_l} (dx)\nonumber\\
=&\int_{\rd} (e^{-\langle{\rm i}\vt,x\rangle}-1)N_{\vsig}(dx).\label{eq:chf2}
\end{align}

It is also easy to see that the measure $N_{\vsig}$ satisfies $\int_{|x|<1}|x|N_{\vsig}(dx)\le N_{\vsig}(\rd)<\infty$.
Thus $f_{\vsig}$ is an infinitely divisible characteristic function and $N_{\vsig}$ is the L\'evy measure of $f_{\vsig}$ by Propositions \ref{pro:LK1} and \ref{pro:LK2}.
As same as the case of $N^{(\va)}_{\vsig}$, regarding $c=N_{\vsig}(\rd)$ and $\rho (dx) =c^{-1}N_{\vsig}(dx)$ in \eqref{CPcf} of Example \ref{examples}, we can see that $\eqref{eq:chf2}$ is of the form of a compound Poisson characteristic function. 
Hence $f_{\vsig}$ is a compound Poisson characteristic function with a finite L\'evy measure $N_{\vsig}$ on $\rd$. 
This completes the proof.
\end{proof}

Now we define another multidimensional zeta distribution generated by $Z_E$.

\begin{definition}[Multidimensional \textit{m}-rank compound Poisson zeta distribution]
Let $\va_{1}, \ldots , \va_{m} $ be $\R^d$-valued vectors satisfying LI or LR and $\alpha_l(p) \ge 0$ for all $1\le l\le m$ and $p\in\Prime$ in \eqref{eq:def1}. 
A distribution on $\rd$ is a multidimensional \textit{m}-rank compound Poisson zeta distribution if it has a characteristic function 
$$
f_{\vsig}\left(\vt\,\right) =\frac{Z_E\left(\vsig +{\rm i}\vt\,\right)}{Z_E(\vsig)}.
$$
\end{definition}

In the following, we use the Dirichlet characters $\chi$ and functions $L_1$ and $L_2$ appeared in Example \ref{ex:41}.

\begin{example}
\pn
$(i)$ A function to be \textit{m}-rank compound Poisson zeta.
$$
\zeta (s_1)\zeta (s_1+s_2),\q \zeta (s_1+\alpha)\zeta (s_1+s_2),\q \alpha >0. 
$$
\pn
$(ii)$ Functions not to generate probability distributions. 
$$
\zeta(s_1)L(s_1+s_2,\chi), \q \zeta(s_1)L_m(s_1+s_2),\q m=1,2.
$$
\end{example}

\subsection{Main theorem}

Historically, we have mentioned that the Riemann zeta distribution is known to be compound Poisson in Proposition \ref{pro:RD}.
Some multidimensional zeta distributions including not infinitely divisible ones are introduced in \cite{AN12s}, \cite{AN11k} and some others.
Still, there are no infinitely divisible multidimensional zeta distributions, which are discrete infinitely divisible with infinitely many mass points, defined yet.
In Sections 3.1 and 3.2, we have studied the relations between one and multidimensional polynomial Euler products and infinite divisibility.
By following these stories, we regard next multidimensional polynomial Euler product is the most suitable form to define a naturally expanded multidimensional compound Poisson zeta distribution with present number theory fully applied.

\begin{definition}[Multidimensional $\eta$-tuple $\varphi$-rank Euler product, $Z^{\eta,\varphi}_E(\vs)$]
Let $d,\varphi,\eta\in\N$ and $\vs\in\mathbb{C}^d$.
For $-1 \le \alpha_{lk}(p) \le 1$ and $\va_l \in {\mathbb{R}}^d$, $1\le l\le \varphi$, $1\le k\le \eta$ and $p\in\Prime$, we define multidimensional $\eta$-tuple $\varphi$-rank Euler product given by
\begin{equation}\label{eq:defbeta}
Z^{\eta,\varphi}_E (\vs) := \prod_p \prod_{l=1}^{\varphi} \prod_{k=1}^{\eta} 
\left( 1 - \alpha_{lk}(p) p^{-\langle \va_l,\vs\rangle} \right)^{-1}.
\end{equation}
\end{definition}
Note that we have $Z^{\eta,\varphi}_E\in\mathcal{Z}_E$ with $m=\varphi\times\eta$ and converges absolutely also in the region $\min_{1\le l\le \varphi}\Re \langle \va_l,\vs\rangle >1$ (see, Definition \ref{def:EP} and the proof of Theorem \ref{th:EPc}). \\

In the following, we regard $f_{\vsig}$ is a normalized function of $Z^{\eta,\varphi}_E\in\mathcal{Z}_E$ that is
$$
f_{\vsig}(\vt)=\frac{Z^{\eta,\varphi}_E(\vsig +{\rm i}\vt)}{Z^{\eta,\varphi}_E(\vsig)}.
$$

\begin{theorem}[{\bf The main theorem}]\label{th:main}
Suppose that $\R^d$-valued vectors $\va_{1}, \ldots , \va_{\varphi}$ satisfy the condition LI or LR, and $\alpha_{lk} (p)= 0$ or \,$\pm1$ for $1\le l \le \varphi$, $1\le k \le \eta$ and $p\in\Prime$ in \eqref{eq:defbeta}. 
Then $f_{\vsig}$ is a characteristic function if and only if $\sum_{k=1}^{\eta} \alpha_{lk} (p) \ge 0$ for all $1\le l \le \varphi$ and $p\in\Prime$.
Moreover, $f_{\vsig}$ is a compound Poisson characteristic function with its finite L\'evy measure $N^{\eta,\varphi}_{\vsig}$ on $\rd$ given by 
\begin{equation}\label{eq:mlm}
N^{\eta,\varphi}_{\vsig} (dx) = \sum_p \sum_{r=1}^{\infty} \sum_{l=1}^{\varphi} \sum_{k=1}^{\eta} \frac{1}{r}
\alpha_{lk} (p)^r p^{-r\langle\va_{l},\vsig\rangle} \delta_{\log p^r \va_{l}} (dx).
\end{equation}
\end{theorem}
Note that the case when $\varphi=1$ is of \textit{m}-tuple and is of \textit{m}-rank when $\eta=1$.

\begin{proof}
First, we show that if there exists a set of pairs of $h$, $1\le h\le \varphi$, and prime numbers $q$ satisfying $\sum_{k=1}^{\eta} \alpha_{lk}(q) < 0$, then there exists $\vt_0\in\rd$ such that $|f_{\vsig}(\vt_0)|>1$. 
As in the proof of Lemma \ref{lm:ln}, for both cases when $\va_{1}, \ldots , \va_{\varphi}$ satisfy the condition LI or LR, we have the existence of $\vt_0\in\rd\backslash\{0\}$ and algebraic real numbers $\om_1,\omega_2, \ldots ,\omega_{\varphi}$ which are linearly independent over the rationals with $\omega_1=1$. 
Thus we can say the following.

Define a vector-valued function $D(T\vt_0)$, $T\in\R$, as follows:
\begin{equation*}
\begin{split}
D(T\vt_0)&:= \log \left|\frac{Z^{\eta,\varphi}_E (\vsig + {\rm{i}}T\vt_0)}{Z^{\eta,\varphi}_E (\vsig)}\right| \\
&=\frac{1}{2} \sum_p \sum_{r=1}^{\infty} \sum_{l=1}^{\varphi} \sum_{k=1}^{\eta} \frac{1}{r} 
\left( \alpha_{lk}(p)^r p^{-r\langle\va_l,\vsig\rangle} 
\bigl(p^{r{\rm{i}}T\langle\va_l,\vt_0\rangle} + p^{-r{\rm{i}}T\langle\va_l,\vt_0\rangle} -2 \bigr)\right) \\
&=\frac{1}{2} \sum_p \sum_{r=1}^{\infty} \sum_{l=1}^{\varphi} \frac{1}{r} 
\biggl( \sum_{k=1}^{\eta} \alpha_{lk}(p)^r \biggl) \Bigl( p^{-r\langle\va_l,\vsig\rangle} 
\bigl(p^{r{\rm{i}}T\omega_l} + p^{-r{\rm{i}}T\omega_l} -2 \bigr)\Bigr) .
\end{split}
\end{equation*}

Now put $\beta_l (p) := \sum_{k=1}^{\eta} \alpha_{lk} (p)$ and $\beta_l^{(r)} (p) := \sum_{k=1}^{\eta} \alpha_{lk} (p)^r$ for $1\le l\le \varphi$, $p\in\Prime$ and $r\in\N$. 
Then, for $j\in\N$, $\beta_l (p)$ and $\beta_{l}^{(2j-1)} (p)$ have the same sign and $\beta_{l}^{(2j)} (p) \ge 0$. 
So that we have $\beta_l (p)\ge 0$ if and only if $\beta_l^{(r)} (p)\ge 0$ for all $r\in\N$.
Replace $\alpha_l (p)^r$ by $\beta_l^{(r)} (p)$, we can show the existence of $\vt_0\in\rd$ such that $|f_{\vsig}(\vt_0)|>1$ by following the proof of Lemma \ref{lm:ln}.
Therefore, to complete the proof, we only have to show that $f_{\vsig}$ is a compound Poisson characteristic function with a finite L\'evy measure $N^{\eta,\varphi}_{\vsig}$ if $\sum_{k=1}^{\eta} \alpha_{lk} (p) \ge 0$ for all $1\le l \le \varphi$ and $p\in\Prime$.

Suppose that $\sum_{k=1}^{\eta} \alpha_{lk} (p) \ge 0$ for all $1\le l \le \varphi$ and $p\in\Prime$.
Then we have 
\begin{align*}
\log f_{\vsig}(\vt)=& \log \frac{Z^{\eta,\varphi}_E (\vsig + {\rm i}\vt)}{Z^{\eta,\varphi}_E (\vsig)} 
= \sum_p \sum_{r=1}^{\infty} \sum_{l=1}^{\varphi} \sum_{k=1}^{\eta}
\frac{1}{r}\alpha_{lk}(p)^r p^{-r\langle\va_l,\vsig\rangle} 
\bigl(p^{-r\langle\va_l,{\rm i}\vt\rangle} -1\bigr)
\nonumber\\ 
=& \sum_p \sum_{r=1}^{\infty} \sum_{l=1}^{\varphi} \sum_{k=1}^{\eta} \frac{1}{r} 
\alpha_{lk}(p)^r p^{-r\langle\va_l,\vsig\rangle} \bigl(e^{-r\langle\va_l,{\rm i}\vt\rangle \log p} -1\bigr)
\nonumber\\ 
= &\int_{\rd} (e^{-\langle{\rm i}\vt,x\rangle}-1) \sum_p \sum_{r=1}^{\infty} \sum_{l=1}^{\varphi} \sum_{k=1}^{\eta}
\frac{1}{r} \alpha_{lk}(p)^r p^{-r\langle\va_l,\vsig\rangle} \delta_{\log p^r \va_l} (dx)\nonumber\\
=&\int_{\rd} (e^{-\langle{\rm i}\vt,x\rangle}-1)N^{\eta,\varphi}_{\vsig}(dx).\label{eq:chfmt1}
\end{align*}
By following the proof of Theorem \ref{th:cfep3}, we can say that $N^{\eta,\varphi}_{\vsig}$ is a finite L\'evy measure on $\rd$ and $f_{\vsig}$ is a compound Poisson characteristic function.
This completes the proof.
\end{proof}

Finally, we define a multidimensional zeta distribution generated by $Z^{\eta,\varphi}_E$.

\begin{definition}[Multidimensional $\eta$-tuple $\varphi$-rank compound Poisson zeta distribution]
Let $\va_{1}, \ldots , \va_{\varphi} $ be $\R^d$-valued vectors satisfying LI or LR, $\sum_{k=1}^{\eta}\alpha_{lk}(p)\ge 0$ with $\alpha_{lk} (p)= 0$ or \,$\pm1$ for $1\le l \le \varphi$, $1\le k \le \eta$ and $p\in\Prime$ in \eqref{eq:defbeta}. 
A distribution on $\rd$ is a multidimensional $\eta$-tuple $\varphi$-rank compound Poisson zeta distribution if it has a characteristic function 
$$
f_{\vsig}\left(\vt\,\right) =\frac{Z^{\eta,\varphi}_E\left(\vsig +{\rm i}\vt\,\right)}{Z^{\eta,\varphi}_E(\vsig)}.
$$
\end{definition}

In the following, we again use functions $L_1$ and $L_2$ appeared in Example \ref{ex:41}.

\begin{example}
\pn
$(i)$ Functions to be $\eta$-tuple $\varphi$-rank compound Poisson zeta.
$$
\zeta (s_1)L_m(s_1)\zeta (s_1+s_2)L_m(s_1+s_2),\q m=1,2. 
$$
\pn
$(ii)$ Functions not to generate probability distributions. 
$$
L_m(s_1)\zeta (s_1+s_2)L_m(s_1+s_2), \q\zeta (s_1)L_m(s_1)L_m(s_1+s_2),\q m=1,2. 
$$
\end{example}

\subsection{Moments}

We have the following.

\begin{theorem}
Let $k\in\N$.
For $Z_E (\vs) \in \mathcal{Z}_E$, if $f_{\vsig} (\vt) = Z_E (\vsig+{\rm{i}}\vt)/Z_E (\vsig)$ is a characteristic function, then the distribution deduced by $f_{\vsig} (\vt)$ has a finite absolute moment of order $2k$.
\label{th:mepmoment}
\end{theorem}

\begin{remark}
Note that this theorem also contains some other functions of $\mathcal{Z}_E$ which were not considered in Sections 3.1, 3.2 and 3.3.
\end{remark}

\begin{proof}
We have, for $1\le l\le m$,
$$
\prod_{p} \bigl( 1 - \alpha_{l}(p) p^{-\langle \va_l,\vs\rangle} \bigr)^{-1} =
\prod_{p} \biggl( 1 + \sum_{k=1}^\infty \frac{\alpha_{l}(p)^k}{p^{k\langle \va_l,\vs\rangle}} \biggr) =
\sum_{n=1}^{\infty} \frac{A_l(n)}{n^{\langle \va_l,\vs\rangle}}, 
\q A_l (n)= \prod_{p|n} \alpha_{l}(p)^{\nu(n;p)}.
$$
This equality coincides with \cite[Lemma 2.2]{Steuding1} when $m=1$.
Note that $|A_l(n)| \le 1$ since $-1 \le \alpha_{l}(p) \le 1$. 
Thus we have
$$
\prod_p \prod_{l=1}^m \bigl( 1 - \alpha_{l}(p) p^{-\langle \va_l,\vs\rangle} \bigr)^{-1} =
\prod_{l=1}^m \sum_{n_l=1}^{\infty} \frac{A_l(n_l)}{n_l^{\langle \va_l,\vs\rangle}} =
\sum_{n_1, \ldots ,n_m=1}^\infty
\frac{A_1(n_1)}{n_1^{\langle \va_1,\vs\rangle}} \cdots \frac{A_m(n_m)}{n_m^{\langle \va_m,\vs\rangle}}.
$$
Obviously, we have $\prod_{l=1}^m |A_l(n_l)| \le 1$. Therefore the series in the formula above converges absolutely. Put $\va_l := (a_{l1}, \ldots, a_{ld})$. For any $1 \le h \le d$, $\varepsilon >0$ and sufficiently large $n_1, \ldots , n_m$ we have
$$
|\log n_1^{-a_{1h}} + \cdots + \log n_m^{-a_{mh}} |\le ( n_1 \cdots n_m )^{\varepsilon} .
$$
Therefore we have
\begin{equation*}
\begin{split}
\frac{\partial}{\partial s_h} Z_E (\vsig+{\rm{i}}\vt) =& \sum_{n_1, \ldots ,n_m=1}^\infty \frac{\partial}{\partial s_h} \frac{A_1(n_1)}{n_1^{\langle \va_1,\vs\rangle}} \cdots \frac{A_m(n_m)}{n_m^{\langle \va_m,\vs\rangle}} \\
= & \sum_{n_1, \ldots ,n_m=1}^\infty
\frac{A_1(n_1)}{n_1^{\langle \va_1,\vs\rangle}} \cdots \frac{A_m(n_m)}{n_m^{\langle \va_m,\vs\rangle}}
\sum_{l=1}^{m} \log n_l^{-a_{lh}}. 
\end{split}
\end{equation*}
Hence the series of $(\partial / \partial s_h) Z_E (\vsig+{\rm{i}}\vt)$ also converges absolutely. 
Inductively, the series of $(\partial^{k_1} / \partial^{k_1} s_1) \cdots (\partial^{k_d} / \partial^{k_d} s_d) Z_E (\vsig+{\rm{i}}\vt)$ converges absolutely, too.
Thus we can complete the proof by Proposition \ref{pro:Sato} (ii) since $f_{\vsig} (\vt)$ is a characteristic function.
\end{proof}

\section{$1$-dimensional important examples}

We have studied multidimensional case of zeta functions and their definable probability distributions on $\rd$.
Though, as in Sections 1.3 and 1.4, what is known about them are still not enough even for $1$-dimensional case.  
In this section, we give some important examples of $1$-dimensional zeta functions whose normalized functions appear to play interesting roles in our story.
Before we mention them, we need the following well-known function which is also mentioned in Section 1.2.

\begin{definition}[Dedekind zeta function of $\mathbb{Q}({\rm i})$ (see, e.g.\,\cite{Cohen})]
Let $\mathbb{Q}({\rm i})$ be a quadratic field of discriminant $-1$ and put
\begin{equation}\label{f:L}
L(s) := \prod_{p \,:\, {\rm{odd}}} \Bigl( 1-(-1)^{\frac{p-1}{2}}p^{-s} \Bigr)^{-1}.
\end{equation}
Then the Dedekind zeta function of $\mathbb{Q}({\rm i})$ is a function of a complex variables $s=\sigma +{\rm i}t$, for $\sigma >1$ given by
$$
\zeta_{{\mathbb{Q}}({\rm i})} (s) := \zeta (s) L(s).
$$
\end{definition}

Now we give the following examples. To understand these examples are not easy, so that we give their proofs in the next subsections.

\begin{example}\label{ex:5}
$(i)$ The Dedekind zeta function generates a \textit{m}-tuple compound Poisson characteristic function. \\
$(ii)$ Let $L(s)$ be the function given in $\eqref{f:L}$. 
For $\sigma >1$, $\zeta (s)^2 L(2s)$ generates an infinitely divisible characteristic function but $L(s) \zeta (2s)$ does not generate even a characteristic function.
\end{example}

\subsection{Proof of Example \ref{ex:5} (i)}

By Theorem \ref{th:cfep}, this belongs to \textit{m}-tuple compound Poisson zeta distribution. 
Though, we write its L\'evy measure explicitly as to compare with which appears in Example \ref{ex:5} (ii). 
We have
\begin{equation*}
\begin{split}
&\log \frac{\zeta_{\Q ({\rm{i}})}(\sigma+{\rm{i}}t)}{\zeta_{\Q ({\rm{i}})}(\sigma)} 
=\log \frac{\zeta(\sigma+{\rm{i}}t) L(\sigma+{\rm{i}}t)}{\zeta(\sigma) L(\sigma)} \\
=&\sum_{r=1}^{\infty} \frac{1}{r} 2^{-r\sigma} \bigl( 2^{-r{\rm{i}}t} -1 \bigr) +
\sum_{p \ge 3} \sum_{r=1}^{\infty} \frac{1}{r} \Bigl( 1+(-1)^{\frac{r(p-1)}{2}} \Bigr) 
p^{-r\sigma} \bigl( p^{-r{\rm{i}}t} -1 \bigr) \\
=&\int_0^{\infty}\left(e^{{\rm i}tx}-1\right)N_{\sigma}(dx),
\end{split}
\end{equation*}
where $N_{\sigma}$ is a finite L\'evy measure on $\R$ given by
$$
N_{\sigma}(dx)=\sum_{r=1}^{\infty} \frac{1}{r} 2^{-r\sigma} \delta_{r\log 2}(dx) + \sum_{p \ge 3} 
\sum_{r=1}^{\infty} \frac{1}{r} \Bigl( 1+(-1)^{\frac{r(p-1)}{2}} \Bigr) p^{-r\sigma}\delta_{r\log p}(dx) .
$$
Thus $\zeta_{\Q ({\rm{i}})}(\sigma+{\rm{i}}t)/\zeta_{\Q ({\rm{i}})}(\sigma)$ is a compound Poisson characteristic function which implies it is a \textit{m}-tuple infinitely divisible characteristic function.

\subsection{Proof of Example \ref{ex:5} (ii)}

First we prove that $f_{\sigma}(t):=\zeta (\sigma +{\rm i}t)^2 L(2\sigma +2{\rm i}t)\zeta (\sigma)^{-2} L(2\sigma)^{-1}$, $\sigma >1, t\in\R$, is an infinitely divisible characteristic function.
We have
\begin{equation*}
\begin{split}
\log f_{\sigma}(t)
=&\sum_{r=1}^{\infty} \frac{2}{r} 2^{-r\sigma} \bigl( 2^{-r{\rm{i}}t} -1 \bigr) +
\sum_{p \ge 3} \sum_{k=1}^{\infty} \frac{2}{2k-1} p^{-(2k-1)\sigma} \bigl( p^{-(2k-1){\rm{i}}t} -1 \bigr) \\
&+ \sum_{p \ge 3} \sum_{k=1}^{\infty} 
\frac{1}{k} \Bigl( 1+(-1)^{\frac{k(p-1)}{2}}\Bigr) p^{-2k\sigma} \bigl( p^{-2k{\rm{i}}t} -1 \bigr)\\
=&\int_0^{\infty}\left(e^{{\rm i}tx}-1\right)N_{\sigma}(dx),
\end{split}
\end{equation*}
where $N_{\sigma}$ is a finite L\'evy measure on $\R$ given by
\begin{equation*}
\begin{split}
N_{\sigma}(dx)=&\sum_{r=1}^{\infty} \frac{2}{r} 2^{-r\sigma} \delta_{r\log 2}(dx) +
\sum_{p \ge 3} \sum_{k=1}^{\infty} \frac{2}{2k-1} p^{-(2k-1)\sigma} \delta_{(2k-1)\log p}(dx) \\
&+\sum_{p \ge 3} \sum_{k=1}^{\infty} \frac{1}{k} \Bigl( 1+(-1)^{\frac{k(p-1)}{2}}\Bigr) p^{-2k\sigma} \delta_{2k\log p}(dx) .
\end{split}
\end{equation*}
Thus $f_{\sigma}$ generated by $\zeta (s)^2 L(2s)$, $\sigma >1$, is a compound Poisson characteristic function 
which implies it is an infinitely divisible characteristic function.\\

Next we prove that $g_{\sigma}(t):=\zeta (2\sigma +2{\rm i}t) L(\sigma +{\rm i}t)\zeta (2\sigma)^{-1} L(\sigma)^{-1}$, $\sigma >1, t\in\R$, is not a characteristic function. Put $D(t) := \log |g_{\sigma}(t)|$. Then one has 
\begin{equation*}
\begin{split}
D(t) =& \,  2^{-2r\sigma -1} \bigl( 2^{2r{\rm{i}}t} + 2^{-2r{\rm{i}}t}- 2 \bigr) \\ &+ 
\frac{1}{2} \sum_{p\ge 3} \sum_{r=1}^\infty \Bigl( p^{-2r\sigma} \bigl( p^{2r{\rm{i}}t} + p^{-2r{\rm{i}}t}- 2 \bigr) + (-1)^{\frac{p-1}{2}r} p^{-r\sigma} \bigl( p^{r{\rm{i}}t} + p^{-r{\rm{i}}t} - 2 \bigr) \Bigr).
\end{split}
\end{equation*}
For any $\ep >0$, we can see that there exists an integer $K$ such that $|2\sum_{r,p>2K}' (p^{-2r\sigma}+p^{-r\sigma})|< \varepsilon$ by the absolute convergence (see, Theorem \ref{th:EPc}). 
Obviously, we have
$$
(-1)^{\frac{p-1}{2}} = 
\begin{cases}
1 & p \equiv 1 \mod 4, \\
-1 & p \equiv 3 \mod 4.
\end{cases}
$$
In the view of $p_k^{{\rm{i}}t}= e^{{\rm{i}}t\log p_k} = e^{2\pi {\rm{i}}t\theta_k}$, \eqref{eq:deftheta} and by Proposition \ref{pro:kroap1}, for any $\ep '>0$ independent of $\ep$ and $K$, there exists $T_0 \in \R$ such that
$$
|p^{{\rm{i}}T_0}-1| < \varepsilon', \q p \in \Prime^+_K, \q \mbox{and} \q 
|q^{{\rm{i}}T_0}+1| < \varepsilon', \q q \in \Prime^-_K,
$$
where $\Prime^+_K := \{ p\in\Prime : p=2 \q \mbox{or} \q 5 \le p \le 2K, \q p \equiv 1 \mod 4 \}$ and $\Prime^-_K := \{ q\in\Prime : 3 \le q \le 2K, \q q \equiv 3 \mod 4 \}$.
By the factorization $x^{2r}-1 = (x+1)(x-1)(x^{2r-2}+x^{2r-4}+\cdots+1)$ when $r\in \N$ and $1 \le r \le 2K$, it holds that $|p^{2r{\rm{i}}T_0}-1| < 2r \varepsilon' \le 4K\varepsilon'$, $p \in \Prime_K$, where $\Prime_K := \{ p \in \Prime : 2 \le p \le 2K \}$. 
Thus by following the proof of Lemma \ref{lem:key}, there exits $T_0\in\R$ such that
\begin{equation*}
\begin{split}
-8K\varepsilon' < p^{2r{\rm{i}}T_0} + p^{-2r{\rm{i}}T_0} -2 \le 0, \q
&p\in\Prime_K, \,\, 1 \le r \le 2K ,\\
-4K\varepsilon' < p^{r{\rm{i}}T_0} + p^{-r{\rm{i}}T_0} -2 \le 0, \q
&p\in\Prime^+_K, \,\, 1 \le r \le 2K ,\\
-4 -4K\varepsilon' < q^{(2k-1){\rm{i}}T_0} + q^{-(2k-1){\rm{i}}T_0} - 2 < -4 + 4K\varepsilon' , \q
&q\in\Prime^-_K, \,\, 1 \le k \le K,\\ 
-4K\varepsilon' \le q^{2k{\rm{i}}T_0} + q^{-2k{\rm{i}}T_0} - 2 \le 0, \q
&q\in\Prime^-_K, \,\, 1 \le k \le K.
\end{split}
\end{equation*}

Therefore we obtain
\begin{equation*}
\begin{split}
D(T_0) >&-\ep -4K\varepsilon' \sum_{p \in \Prime} \sum_{r=1}^{2K} \frac{1}{r} p^{-2r\sigma} -
2K\varepsilon' \sum_{p \in \Prime^+_K} \sum_{r=1}^{2K} \frac{1}{r} p^{-r\sigma} \\
& + (2K\varepsilon' -2)\sum_{q \in \Prime^-_K}\sum_{k=1}^{K} \frac{(-1)^{2k-1}}{2k-1} q^{-(2k-1)\sigma}
-2K\varepsilon' \sum_{q \in \Prime^-_K}\sum_{k=1}^{K} \frac{(-1)^{2k}}{2k} q^{-2k\sigma}.
\end{split}
\end{equation*}
Suppose $\varepsilon$ is sufficiently small and $\varepsilon'$ such that $KC'\varepsilon' < \varepsilon$. 
Then we have $D(T_0) >0$ by
$$
-2 \sum_{q \in \Prime^-_K}\sum_{k=1}^{K} \frac{(-1)^{2k-1}}{2k-1} q^{-(2k-1)\sigma} >0. 
$$
This completes the proof.\\

Throughout this paper, we have considered when multivariable Euler products to generate infinitely divisible zeta distributions on $\rd$ 
since it is difficult to obtain them by series representations studied in \cite{AN12s}.
However, they also include products which generate not infinitely divisible $\rd$-valued characteristic functions and not even to generate 
characteristic functions.
Their properties seem to be interesting but rather difficult to be treated by our multidimensional polynomial Euler product.
To obtain more detail of behaviors of products in this view, 
we have studied them by treating multivariable finite Euler products, which are simple cases of our products, in \cite{AN12q}.

\section*{Acknowledgment}
The second author was partially supported by JSPS grant 16K05077.

 

\begin{thebibliography}{1}

\bibitem{AN11k}
{\rm T.~Aoyama \and T.~Nakamura}, `Zeros of zeta functions and zeta distributions on $\rd$', Functions in Number Theory and Their Probabilistic Aspects, {\it{RIMS K\^oky\^uroku Bessatsu}}, B34, (2012) 39--48. 

\bibitem{AN12q}
{\rm T.~Aoyama \and T.~Nakamura}, `Behaviors of multivariable finite Euler products in probabilistic view',  {\em Math. Nachr.} 
286 (2013) 1691--1700.

\bibitem{AN12s}
{\rm T.~Aoyama \and T.~Nakamura}, `Multidimensional Shintani zeta functions and zeta distributions on $\rd$',  {\em Tokyo J. Math.} 36 (2013) 521--538.

\bibitem{Apo}
{\rm T.~M.~Apostol}, 
{\em Introduction to Analytic Number Theory} (Undergraduate Texts in Mathematics, Springer, 1976). 

\bibitem{Apo2}
{\rm T.~M.~Apostol}, {\em Modular functions and Dirichlet series in Number Theory } (Graduate Texts in Mathematics 41, Springer, 1990). 

\bibitem{Baker}
{\rm A.~Baker}, {\em Transcendental number theory } (Cambridge Mathematical Library, Cambridge University Press, Cambridge, 1975). 

\bibitem{Bho} 
{\rm G.~Bhowmik, D.~Essouabri \and B.~Lichtin}, `Meromorphic continuation of multivariable Euler products', {\em ForumMath. }19 no. 6 (2007) 1111--1139.

\bibitem{BY}
{\rm P.~Biane, J.~Pitman and M.~Yor}, `Probability laws related to the Jacobi theta and Riemann zeta functions, and Brownian excursions', 
{\em Bull. Amer. Math. Soc.} 38 (2001) 435--465.

\bibitem{Yor}
{\rm P.~Bourgade, C.~Hughes, A.~Nikeghbali, M.~Yor}, `The characteristic polynomial of a random unitary matrix: a probabilistic approach', 
{\em Duke Math. J.} 145 (2008) 45--69, 

\bibitem{Cohen} 
{\rm H.~Cohen}, {\em Number theory. Vol.~II. Analytic and modern tools} (Graduate Texts in Mathematics, 240. Springer, New York, 2007). 

\bibitem{Conway}
J.~B.~Conway, {\it{Functions of one complex variable I}}. Second edition. Graduate Texts in Mathematics, 11. Springer-Verlag, New York-Berlin, 1978. 

\bibitem{GK68}
{\rm B.~V.~Gnedenko \and A.~N.~Kolmogorov}, {\em Limit Distributions for Sums of Independent Random Variables (Translated from the Russian by
Kai Lai Chung)} (Addison-Wesley, 1968).

\bibitem{Hu06}
{\rm C.-Y.~Hu, A.~M.~Iksanov, G.~D.~Lin \and O.~K.~Zakusylo}, `The Hurwitz zeta distribution', {\em Aust. N.~Z.~J.~Stat. }48 (2006) 1--6.

\bibitem{JW35}
{\rm B.~Jessen \and A.~Wintner}, `Distribution Functions and the Riemann Zeta Function', {\em Trans. Amer. Math. Soc.} 38 (1935), no. 1, 48--88.

\bibitem{Khi}
{\rm A.~Ya.~Khinchine}, {\em Limit Theorems for Sums of Independent Random Variables (in Russian)}, (Moscow and Leningrad, 1938).

\bibitem{Lin}
{\rm G.~D.~Lin \and C.-~Y.~Hu}, `The Riemann zeta distribution', {\em Bernoulli }7 (2001) 817--828.

\bibitem{Nakamurasr1} 
{\rm T.~Nakamura}, `The joint universality and the generalized strong recurrence for Dirichlet $L$-functions', {\em Acta Arith. }138 no.~4 (2009) 357--362. 

\bibitem{S99} 
{\rm K.~Sato}, {\em L\'evy Processes and Infinitely Divisible Distributions} (Cambridge University Press, 1999).

\bibitem{Selberg}
{\rm A.~Selberg}, `Old and new conjectures and results about a class of Dirichlet series', {\em Proceedings of the Amalfi Conference on Analytic Number Theory (Maiori, 1989)  Univ. Salerno, Salerno} (1992) 367--385. 

\bibitem{Steuding1} 
{\rm J.~Steuding}, {\em Value-Distribution of L-functions} (Lecture Notes in Mathematics, 1877, Springer, Berlin, 2007).

\bibitem{Tit}
{\rm E.~C.~Titchmarsh}, {\em The theory of the Riemann zeta-function. Second edition. Edited and with a preface by D.~R.~Heath-Brown} (The Clarendon Press, Oxford University Press, New York, 1986). 

\end{thebibliography}
\end{document}